\documentclass[11pt]{amsart}

\usepackage{amsmath, amsfonts, amsthm, amssymb, array, xcolor, graphicx}
\usepackage{lscape}
\usepackage{booktabs}
\usepackage{mathrsfs}
\usepackage{mathtools}

\usepackage{fullpage}

\newtheorem{theorem}{Theorem}[section]
\newtheorem{conjecture}[theorem]{Conjecture}
\newtheorem{corollary}[theorem]{Corollary}

\newtheorem{lemma}[theorem]{Lemma}

\numberwithin{theorem}{section}

\newcommand{\N}{\mathbb{N}}

\def \suchthat {\text{ } | \text{ } }
\def \N {\mathbb{N}}

\begin{document}

\title{Generalized Alder-Type Partition Inequalities}

\author{Liam Armstrong}
\address{Oregon State University}
\email{armstrli@oregonstate.edu}
\author{Bryan Ducasse}
\address{University of Central Florida}
\email{bducasse77@knights.ucf.edu}
\author{Thomas Meyer}
\address{Amherst College}
\email{tmeyer23@amherst.edu}
\author{Holly Swisher}
\address{Oregon State University}
\email{swisherh@oregonstate.edu}

\thanks{This work was supported by NSF grant DMS-2101906 via the 2022 Oregon State University REU in Number Theory.}



\begin{abstract}
In 2020, Kang and Park conjectured a ``level $2$" Alder-type partition inequality which encompasses the second Rogers-Ramanujan Identity.  Duncan, Khunger, the fourth author, and Tamura proved Kang and Park's conjecture for all but finitely many cases utilizing a  ``shift" inequality and conjectured a further, weaker generalization that would extend both Alder's (now proven) as well as Kang and Park's conjecture to general level.  Utilizing a modified shift inequality, Inagaki and Tamura have recently proven that the Kang and Park conjecture holds for level $3$ in all but finitely many cases.  They further conjectured a stronger shift inequality which would imply a general level result for all but finitely many cases. Here, we prove their conjecture for large enough $n$, generalize the result for an arbitrary shift, and discuss the implications for Alder-type partition inequalities.
\end{abstract}

\maketitle

\section{Introduction}
A partition of a positive integer $n$ is a non-increasing sequence of positive integers, called parts, that sum to $n$.   Let $p(n \, | \text{ condition})$ count the number of partitions of $n$ that satisfy the specified condition, and define
\begin{align*}
q_d^{(a)}(n) &:= p(n\, | \text{ parts} \geq a \text{ and differ by at least } d), \\
Q_d^{(a)}(n)  &:= p(n\, | \text{ parts} \equiv \pm a \!\!\! \pmod{d+3}), \\
\Delta_d^{(a)}(n) &:= q_d^{(a)}(n)-Q_d^{(a)}(n).
\end{align*}
Euler's well-known partition identity, which states that the number of partitions of $n$ into distinct parts equals those into odd parts, can be written as $\Delta_1^{(1)}(n)=0$.  Moreover, the celebrated first and second Rogers-Ramanujan identities, written here in terms of $q$-Pochhammer notation\footnote{$(a;q)_0:=1$ and $(a;q)_n:=\prod_{k=0}^{n-1}(1-aq^k)$ for $1\leq n\leq \infty$},
\begin{align*}
\sum_{n=0}^\infty\frac{q^{n^2}}{(q;q)_n} &=\frac{1}{(q;q^5)_\infty(q^4;q^5)_\infty},\\
\sum_{n=0}^\infty\frac{q^{n^2+n}}{(q;q)_n} &=\frac{1}{(q^2;q^5)_\infty(q^3;q^5)_\infty},
\end{align*}
are interpreted in terms of partitions as $\Delta_2^{(1)}(n)=0$ and $\Delta_2^{(2)}(n)=0$, respectively.

Schur \cite{SCHUR} proved that the number of partitions of $n$ into parts differing by at least $3$, where no two consecutive multiples of 3 appear, equals the number of partitions of $n$ into parts congruent to $\pm1\pmod{6}$, which yields that $\Delta_3^{(1)}(n)\geq 0$.  Lehmer \cite{Lehmer} and Alder \cite{Alder_nonex} proved that such a pattern of identities can not continue by showing that no other such partition identities can exist.  However, in 1956 Alder \cite{Alder_conj} conjectured a different type of generalization.  Namely, that for all $n,d\geq 1$,
\begin{equation}\label{Alder_conj}
\Delta_d^{(1)}(n)\geq 0.
\end{equation}

In 1971, Andrews \cite{Andrews} proved \eqref{Alder_conj} when $d=2^k-1$ and $k\geq 4$, and in 2004, Yee \cite{Yee_pnas,Yee} proved \eqref{Alder_conj} for $d\geq 32$ and $d=7$, both using $q$-series and combinatorial methods.  Then in 2011, Alfes, Jameson, and Lemke Oliver \cite{AJLO} used asymptotic methods and detailed computer programming to prove the remaining cases of $4\leq d\leq 30$ with $d\neq 7,15$.

It is natural to ask whether \eqref{Alder_conj} can be generalized to $a=2$ in order to encapsulate the second Rogers-Ramanujan identity, or perhaps even be generalized to arbitrary $a$.  

In 2020, after observing that $\Delta_d^{(2)}(n)\geq 0$ does not hold for all $n,d\geq 1$, Kang and Park \cite{KANG-PARK} defined
\begin{align*}
Q_d^{(a,-)}(n)  &:= p(n\, | \text{ parts} \equiv \pm a \!\!\! \pmod{d+3},\text{ excluding the part } d+3-a), \\
\Delta_d^{(a,-)}(n) &:= q_d^{(a)}(n)-Q_d^{(a,-)}(n),
\end{align*}
and conjectured that for all $n,d\geq 1$,
\begin{equation}\label{KP_conj}
\Delta_d^{(2,-)}(n)\geq 0.
\end{equation}

Kang and Park \cite{KANG-PARK} proved \eqref{KP_conj} when $n$ is even, $d=2^k-2$, and $k\geq 5$ or $k=2$.   Then in 2021, Duncan, Khunger, the fourth author, and Tamura \cite{DKST} proved \eqref{KP_conj} for all $d\geq 62$.  Exploring the question for larger $a$, they conjectured that for all $n,d\geq 1$,
\begin{equation}\label{DKST_conj}
\Delta_d^{(3,-)}(n)\geq 0,
\end{equation}
but found that when $a\geq 4$, the removal of one additional part appears to be both necessary and sufficient to obtain such a result for all $n,d\geq 1$.  Letting 
\begin{align*}
Q_d^{(a,-,-)}(n)  &:= p(n\, | \text{ parts} \equiv \pm a \!\!\! \pmod{d+3},\text{ excluding the parts } a \text{ and } d+3-a), \\
\Delta_d^{(a,-,-)}(n) &:= q_d^{(a)}(n)-Q_d^{(a,-,-)}(n),
\end{align*}
Duncan et al. \cite{DKST} conjectured that if $a,d\geq 1$ such that $1 \leq a \leq d+2$, then for all $n \geq 1$,
\begin{equation}\label{gen_dkst}
\Delta_d^{(a,-,-)}(n) \geq 0.
\end{equation}

Recently, Inagaki and Tamura \cite{Inagaki-Tamura} proved \eqref{DKST_conj} for $d\geq 187$ and $d=1,2,91,92,93$, and further proved that $\Delta_d^{(4,-)}(n)\geq 0$ for $d\geq 249$ and $121\leq d\leq 124$ as a corollary to a result for general $a$ for certain residue classes of $d$.  Inagaki and Tamura \cite{Inagaki-Tamura} were also able to prove the general conjecture \eqref{gen_dkst} of Duncan et al. \cite{DKST} for sufficiently large $d$ with respect to $a$, namely when $\lceil \frac{d}{a} \rceil \geq 2^{a+3}-1$.

The proof of \eqref{KP_conj} for $d\geq 62$ by Duncan et al. \cite{DKST} utilized a particular shift identity.  Namely, they showed that if $d\geq 31$ or $d=15$, then for  $n \geq 1$,
\begin{equation}\label{shift_2}
q_d^{(1)}(n) \geq Q_{d-2}^{(1,-)}(n).
\end{equation}
The proof of \eqref{DKST_conj} for $d\geq 187$ or $d=1,2,91,92,93$ by Inagaki and Tamura \cite{Inagaki-Tamura} utilized a stronger shift identity that holds for large enough $n$ with respect to $d$.  Namely, they showed that if $d\geq 63$ or $d=31$, then for  $n \geq d+2$, 
\begin{equation}\label{shift_3}
q_d^{(1)}(n) \geq Q_{d-3}^{(1,-)}(n).
\end{equation} 

Given a choice of $a$, it is natural to ask for which $n,d\geq 1$,
\begin{equation}\label{question}
\Delta_d^{(a,-)}(n)\geq 0.
\end{equation}
Inagaki and Tamura \cite{Inagaki-Tamura} posed the following shift identity conjecture, which they further determined can be used to obtain answers to \eqref{question} and a vast improvement on the bounds for \eqref{gen_dkst}.  
\begin{conjecture}[Inagaki, Tamura \cite{Inagaki-Tamura}, 2022]\label{IT_conj}
Let $d\geq 12$ and $n\geq d+2$. Then
\[q_d^{(1)}(n)-Q_{d-4}^{(1,-)}(n)\geq 0.\]
\end{conjecture}

In this paper, we prove a generalized shift identity.  We have the following theorem.
\begin{theorem}\label{biglemon}
If $N\geq 2$, $d \geq \max\{63,46N-79 \}$, and $n \geq d+2$, then
\[ q_d^{(1)}(n) \geq Q_{d-N}^{(1,-)}(n). \]
\end{theorem}
As an immediate corollary of Theorem \ref{biglemon} we obtain Conjecture \ref{IT_conj} in the cases when $d \geq 105$. 


\begin{corollary}\label{littlelemon}
For $d \geq 105$, and $n \geq d+2$, 
\[ q_d^{(1)}(n) \geq Q_{d-4}^{(1,-)}(n). \]
\end{corollary}

Moreover, using the methods of Inagaki and Tamura \cite{Inagaki-Tamura} Corollary \ref{littlelemon} can be applied to obtain a more complete answer to \eqref{question} as well as stronger bounds for \eqref{gen_dkst}. 

\begin{theorem}\label{thm:gen_kp}
Let $a\geq 1$ and $d\not\equiv -3 \pmod{a}$ such that $\left\lceil\frac{d}{a}\right\rceil \geq 105$.  Then for all $n\geq1$,
\[
\Delta_d^{(a,-)}(n)\geq 0.
\]
Moreover, for $d\equiv -3 \pmod{a}$ then $\Delta_d^{(a,-)}(n)\geq 0$ for all $n\neq d+a+3$.
\end{theorem}

As a corollary of Theorem \ref{thm:gen_kp} we obtain the following, which proves conjecture \eqref{gen_dkst} of Duncan et al. \cite{DKST} for $\left\lceil\frac{d}{a}\right\rceil \geq 105$.  We note that this bound is lower than that given by Inagaki and Tamura \cite[Thm. 1.8]{Inagaki-Tamura} when $a\geq 4$, and is significantly lower as $a$ grows.  

\begin{corollary}\label{thm:gen_dkst}
For all $a,d \geq 1$ such that $\left\lceil\frac{d}{a}\right\rceil \geq 105$, and $n\geq1$,
\[
\Delta_d^{(a,-,-)}(n)\geq 0.
\]
\end{corollary} 
 
We now outline the rest of the paper. In Section \ref{prelemonaries}, we state a fundamental result of Andrews \cite{Andrews} and discuss some notation and lemmas used in the proofs of Theorems \ref{biglemon}, \ref{thm:gen_kp}, and \ref{thm:gen_dkst}. In Section \ref{biglemmaproof}, we prove Theorem \ref{biglemon}, and in Section \ref{conclusion}, we use Corollary \ref{littlelemon} to prove Theorem \ref{thm:gen_kp} and Corollary \ref{thm:gen_dkst}.  We conclude with additional remarks and discussion.

\section{Preliminaries}\label{prelemonaries}

For a nonempty set $A\subseteq \mathbb{N}$, define $\rho(A;n)$ to count the number of partitions of $n$ with parts in $A$.  The following theorem of Andrews \cite{Andrews} gives a way to compare the number of partitions of $n$ with parts coming from different sets.

\begin{theorem}[Andrews \cite{Andrews}, 1971] \label{st-thm}
Let $S = \{x_i\}_{i = 1}^\infty$ and $T = \{y_i\}_{i = 1}^\infty$ be two strictly increasing sequences of positive integers such that $y_1 = 1$ and $x_i \geq y_i$ for all $i$. Then
\[
\rho(T;n) \geq \rho(S;n).
\]
\end{theorem}

For fixed $d\geq 1$, define $r$ to be the greatest integer such that 
\begin{equation}\label{Def: r}
    2^r-1\leq d.
\end{equation}
Further define for integers $d,s \geq 1$
\begin{equation}\label{T-def}
    T_{s, d} := \{y \in \N \suchthat y \equiv 1, d+2,\dots, d+2^{s-1}\!\!\!\!\!\pmod{2d} \}.
\end{equation}

\begin{lemma}\label{hand-waving-made-formal}
Let $d\geq 1$.  If $1\leq a\leq b\leq r$, where $r$ is defined as in \eqref{Def: r}, then $\rho(T_{a,d};n) \leq \rho(T_{b,d};n)$.
\end{lemma}

\begin{proof}

When $s \leq r$, we have $2^{s-1}-1 < d$ which implies that $(2k-1)d+2^{s-1} < 2kd+1$ for all $k\geq 1$.  Thus Table \ref{Ts-table} below shows the elements of $T_{s,d}$ listed in increasing order when read left to right.

\renewcommand{\arraystretch}{1.2}
\begin{table}[ht]
\caption{Elements of $T_{s,d}$ in increasing order by rows for $s\leq r$.}\label{Ts-table}
    \begin{tabular}{|c|c|c|c|}
        \hline
         1&$d+2$&$\cdots$&$d+2^{s-1}$ \\
        \hline
         $2d+1$&$3d+2$&$\cdots$&$3d+2^{s-1}$ \\ 
        \hline
        \vdots&\vdots&\vdots&\vdots\\
        \hline
        $(2j-2)d+1$&$(2j-1)d+2$&$\cdots$&$(2j-1)d+2^{s-1}$ \\
        \hline
        \vdots&\vdots&\vdots&\vdots\\
        \hline
    \end{tabular}
\end{table}

Let $y_i^s$ denote the $i^{\text{th}}$ smallest element of $T_{s,d}$.  Observe that when $1\leq a\leq b\leq r$ we must have that $y_i^a \geq y_i^b$ for all $i$, since the number of columns in Table \ref{Ts-table}, and thus the index of the elements in the first column, is weakly increasing when $s=a$ is replaced by $s=b$.  Thus, by Theorem \ref{st-thm}, we conclude that $\rho(T_{a, d};n) \leq \rho(T_{b,d};n)$. 
\end{proof}

Previous work of Andrews \cite{Andrews} and Yee \cite{Yee} on Alder's conjecture gives the following lower bound for $q_d^{(1)}(n)$ for sufficiently large $d$ and $n$. 

\begin{lemma}[Andrews \cite{Andrews}, Yee \cite{Yee}] \label{lem: qd-greater-than-t5d}
Let $d\geq 63$ and $n\geq 5d$. Then $q_d^{(1)}(n)\geq \rho(T_{5,d};n)$.
\end{lemma}

\begin{proof} 
Recall for fixed $d\geq 1$, $r$ is defined as in \eqref{Def: r}.  When $d > 2^r-1$ for $r\geq 5$ and $n\geq 4d + 2^r$, work of Yee [\cite{Yee}, Lemmas 2.2 and 2.7] gives that
\begin{align*}
q_d^{(1)}(n)&\geq \mathcal{G}_d^{(1)}(n),
\end{align*}
where
\begin{align*}
\sum_{k\geq 0}\mathcal{G}_d^{(1)}(n)q^n=\frac{(-q^{d+2^{r-1}};q^{2d})_\infty}{(q;q^{2d})_\infty(q^{d+2};q^{2d})_\infty\cdots(d^{d+2^{r-2}};q^{2d})_\infty}.
\end{align*}
From this generating function it follows that $\mathcal{G}_d^{(1)}(n)$ counts the number of partitions of $n$ into distinct parts congruent to $d+2^{r-1}$ modulo $2d$ and unrestricted parts from the set $T_{r-1, d}$ as defined in \eqref{T-def}. Thus it follows that
\begin{align*}
    q_d^{(1)}(n)&\geq \mathcal{G}_d^{(1)}(n)\geq \rho(T_{r-1,d};n).
\end{align*}
From our hypotheses $d\geq 63$, so $r\geq 6$.  Hence by Lemma \ref{hand-waving-made-formal}, we have when $d > 2^r-1$ that
\begin{align*}
    q_d^{(1)}(n)&\geq \rho(T_{5,d};n),
\end{align*}
as desired.

When $d=2^r-1$ for $r\geq 4$, work of Andrews [\cite{Andrews}, Theorem 1 and discussion], gives that 
\begin{align*}
    q_d^{(1)}(n)&\geq \mathcal{L}_d(n),
\end{align*}
where
\[
\sum_{n\geq 0}\mathcal{L}_d(n)q^n =\frac{1}{(q;q^{2d})_\infty(d^{d+2};q^{2d})_\infty\cdots(q^{d+2^{r-1}};q^{2d})_\infty}.
\]
From this generating function it follows that $\mathcal{L}_d(n) = \rho(T_{r,d};n)$.  Thus with our hypotheses, and Lemma \ref{hand-waving-made-formal}, it follows that when $d = 2^r-1$,
\[
    q_d^{(1)}(n)\geq \rho(T_{5,d};n),
\]
as desired.
\end{proof}

Let
\[
S_d^N := \{x \in \mathbb{N} \mid x \equiv \pm1 \!\!\!\!\! \pmod{d-N+3}\} \setminus \{d-N+2\},
\]
so that we have by definition
\begin{equation}\label{QtoS}
Q_{d-N}^{(1,-)}(n) = \rho(S_d^N; n).
\end{equation}
We write $x_i^N$ and $y_i$ to denote the $i^{\text{th}}$ smallest elements of $S_{d}^N$ and $T_{5,d}$, respectively. 

If $x_i^N \geq y_i$ for all $i$, then Theorem \ref{biglemon} would follow easily from Theorem \ref{st-thm} and Lemma \ref{lem: qd-greater-than-t5d}.  While this is not the case, the inequality does hold for all but the index $i=2$, as shown in the following lemma.

\begin{lemma} \label{lem: xi-yi-nonnegative}
If $N \geq 2$ and $d \geq \max \{ 31, 6N-17 \}$, then $x_i^N - y_i \geq 0$ for all $i \geq 3$.  Moreover, we have that 
\[
\min_{i \geq 3} \{x_i^N - y_i \} = \min \{d-2N-1, d-6N+17 \}.
\]
\end{lemma}
        
\begin{proof}
Fix $d\geq 1$.  We first show that we can reduce the indices modulo $10$ in our comparison.  By definition of $S_d^N$, we see that for $i\geq 3$, $x_i^N = \lceil \frac{i}{2} \rceil (d-N+3) + (-1)^i$, so it follows that $x_{i+10}^N = x_i^N + 5d - 5N + 15$.  Since $d\geq 31$ we have that $r\geq 5$.  Thus recalling Table \ref{Ts-table}, we can write $y_{i+10} = y_i + 4d$ for all $i\geq 1$.  Thus for $i\geq 3$, we have 
\begin{equation}\label{mod10} 
x_{i+10}^N - y_{i+10} = (x_i^N - y_i) + (d - 5N + 15) \geq x_i^N -y_i, 
\end{equation}
since $d \geq \max \{31, 6N-17 \} \geq 5N - 15$ when $N\geq 2$.

Thus, it suffices to show $x_i^N - y_i \geq 0$ for the indices $3 \leq i \leq 12$. By direct computation, we see
\begin{align*}
    &x_3^N-y_3=d-2N+1,\\
    &x_4^N-y_4=d-2N-1,\\
    &x_5^N-y_5=2d-3N-8,\\
    &x_6^N-y_6=d-3N+9,\\
    &x_7^N-y_7=d-4N+9,\\
    &x_8^N-y_8=d-4N+9,\\
    &x_9^N-y_9=2d-5N+6,\\
    &x_{10}^N-y_{10}=2d-5N,\\
    &x_{11}^N-y_{11}=2d-6N+16,\\
    &x_{12}^N-y_{12}=d-6N+17,
\end{align*}
so that $x_i^N - y_i \geq 0$ when
\[ d \geq \max\{31, 5N - 15, 2N - 1, 2N + 1, \frac{3N+8}{2}, 3N-9, \dots, 3N - 8, 6N - 17 \}. \]
Among these terms, $31$ is maximal when $N \leq 8$ and $6N - 17$ is maximal for $N \geq 8$, so that $x_i^N - y_i \geq 0$ for $d \geq \max \{ 31, 6N - 17 \}$.  Moreover from \eqref{mod10} we have that
\[
\min_{i \geq 3} \{ x_i^N - y_i \} = \min_{3 \leq i\leq 12} \{ x_i^N - y_i \}.
\]
By direct computation we see that among the terms $x_i^N - y_i$ for $3\leq i \leq 12$ listed above, $d-2N-1$ is minimal when $N \leq 4$ and $d-6N+17$ is minimal when $N \geq 5$. Thus
\[ 
\min_{i \geq 3} \{ x_i^N - y_i \} =
\begin{cases}
d-2N-1 & N \leq 4 \\
d-6N+17 & N \geq 5. \\
\end{cases}
\]
\end{proof}

For fixed $d,n\geq 1$, write $S^N$ to denote the set of partitions of $n$ with parts in $S_d^N$ so that $|S^N| = \rho(S_d^N; n)$.  For $\lambda\in S^N$, let $p_i$ denote the number of times $x_i^N$ occurs as a part in $\lambda$, and define
\begin{equation}\label{alpha}
\alpha = \alpha(\lambda) :=\sum_{i \geq 3}(x_i^N-y_i)p_i. 
\end{equation}

The following lemma gives a lower bound on the number of parts equal to $x_2^N=d-N+4$ for certain partitions $\lambda\in S^N$.  It is imperative to our proof of Theorem \ref{biglemon}.

\begin{lemma} \label{lem: p2-geq-8}
Let $N \geq 2$, $d \geq \max \{ 31, 9N-13, 13N-31 \}$, $n \geq 7d+14$, and $\lambda \in S^N$ such that $p_1+\alpha < (N-2)p_2$. Then $p_2 \geq 8$.

\begin{proof}
Suppose $p_2 \leq 7$.  We first observe that if $\alpha \neq 0$, then there exists some $i\geq 3$ such that $p_i\neq 0$.  By Lemma \ref{lem: xi-yi-nonnegative} and our bounds on $d$ it follows that 
\[
\alpha \geq \min \{d-2N-1, d-6N+17 \} \geq 7N - 14.
\]
But then 
\[
p_1+\alpha \geq 7N - 14 \geq (N-2)p_2,
\] 
which contradicts our hypothesis on $p_1$. 

However, if $\alpha =0$, then $p_i=0$ for all $i\geq 3$, and $p_1 < 7N-14$, so 
\[
n=p_1+p_2(d-N+4) < (7N-14) + 7(d-N+4)=7d+14,
\]
which contradicts our hypothesis on $n$.  Thus we must have $p_2 \geq 8$ as desired.
\end{proof}

\end{lemma}

We conclude this section with a few results that will be used in Section \ref{conclusion}.  The first two are lemmas from work of Duncan et al. \cite{DKST} which give key inequalities in our proof of Theorem \ref{thm:gen_kp}.

\begin{lemma}[Duncan et al. \cite{DKST}, 2021]\label{Lem: st-thm with ceiling}
Let $a,d\geq 1$, and let $n\geq d+2a$. Then
\[q_d^{(a)}(n)\geq q_{\left\lceil\frac{d}{a}\right\rceil}^{(1)}\left(\left\lceil\frac{n}{a}\right\rceil\right).\]
\end{lemma}

\begin{lemma}[Duncan et al. \cite{DKST}, 2021]\label{a-to-1}
Let $a,d,n\geq 1$ be such that $a\mid (d+3)$. Then
\begin{align*}
    &Q_d^{(a,-)}(an)=Q_{\frac{d+3}{a}-3}^{(1,-)}(n).
\end{align*}
\end{lemma}

Inagaki and Tamura \cite{Inagaki-Tamura} expanded Theorem \ref{st-thm} to allow for partitions of different integers, which enables us to prove another key inequality in our proof of Theorem \ref{thm:gen_kp}.

\begin{lemma}[Inagaki and Tamura \cite{Inagaki-Tamura}]\label{IT-modiefied-ST}
Let $a\geq 1$, and let $S=\{x_i\}_{i=1}^\infty$ and $T=\{y_i\}_{i=1}^\infty$ be two strictly increasing sequences of positive integers such that $y_1=a$ and $a\mid y_i$, $x_i\geq y_i$ for all $i\geq 1$. Then for all $n\geq 1$,
\[\rho(T;n+\hat{n}_a)\geq \rho(S;n),\]
where $\hat{n}_a$ denotes the least nonnegative integer such that $a\mid (n+\hat{n}_a)$.
\end{lemma}

\section{Proof of Theorem \ref{biglemon}}\label{biglemmaproof}

In this section, we modify the work of Inagaki and Tamura \cite{Inagaki-Tamura} and use results from Andrews \cite{Andrews} and Yee \cite{Yee} to prove Theorem \ref{biglemon}.   As our primary method works only when $n\geq 7d+14$, we first consider the case when $d+2 \leq n \leq 7d+13$ below.

\begin{lemma}\label{bigsmalllemon}
Let $N \geq 2$ and $d \geq \max\{63,46N-79 \}$.  Then for all $d+2 \leq n \leq 7d+13$,
\[
q_d^{(1)}(n) \geq Q_{d-N}^{(1,-)}(n). 
\]
\end{lemma}

\begin{proof}
Observe that $q_d^{(1)}(n)$ and $Q_{d-N}^{(1,-)}(n)$ are both weakly increasing functions since every partition of $n$ counted by $q_d^{(1)}(n)$ or $Q_{d-N}^{(1,-)}(n)$, respectively, injects to a partition of $n+1$ counted by $q_d^{(1)}(n+1)$ or $Q_{d-N}^{(1,-)}(n+1)$, respectively by adding $1$ to the largest part or adding a part of size $1$, respectively.  Thus, if $q_d^{(1)}(k_1) \geq Q_{d-N}^{(1,-)}(k_2)$ for integers $k_1 \leq k_2$, it follows that $q_d^{(1)}(n) \geq Q_{d-N}^{(1,-)}(n)$ for all $k_1 \leq n \leq k_2$.  By our hypotheses on $d$, it follows that $d+2 \leq 2d-2N+4$, $2d-2N+5 \leq 5d-5N+16$, and $5d-5N+17 \leq 7d+13$.  Thus it suffices to prove the following three inequalities.
\begin{align}
q_d^{(1)}(d+2) & \geq Q_{d-N}^{(1,-)}(2d-2N+4), \label{case1} \\
q_d^{(1)}(2d-2N+5) & \geq Q_{d-N}^{(1,-)}(5d-5N+16), \label{case2} \\
q_d^{(1)}(5d-5N+17) & \geq Q_{d-N}^{(1,-)}(7d+13). \label{case3} 
\end{align}
 
Note that the partition $n$ itself is always counted by $q_d^{(1)}(n)$, and for any $1 \leq k \leq \left \lfloor \frac{n-d}{2} \right \rfloor$, the partition $(n-k) + k$ is counted by $q_d^{(1)}(n)$ since then $(n-k) - k \geq d$. Thus, for any $d, n \geq 1$,

\begin{equation}\label{q-lemma}
q_d^{(1)}(n) \geq \max \left\{ 1,\left\lfloor \frac{n-d}{2} \right\rfloor + 1 \right\}.
\end{equation} 

We first prove \eqref{case1}.  Observe that any partition counted by $Q_{d-N}^{(1,-)}(2d-2N+4)$ can only use the parts $x_1^N=1$ and $x_2^N=d-N+4$ since $x_3^N>2d-2N+4$.  There is exactly one such partition with largest part $x_1^N$, and one with largest part $x_2^N$.  Thus $Q_{d-N}^{(1,-)}(2d-2N+4)=2$.  Using \eqref{q-lemma} we obtain that $q_d^{(1)}(d+2) \geq 2$ which gives \eqref{case1}.

We next prove \eqref{case2}.  Since $x^N_{10}=5d-5N+16$, any partition counted by $Q_{d-N}^{(1,-)}(5d-5N+16)$ can only use the parts $x^N_i$ with $1 \leq i \leq 10$.  Using that fact that $d \geq \max\{63,46N-79 \}$, one can calculate that the number of partitions of $5d-5N+16$ with largest part $x_i^N$ as $i$ ranges from $1$ to $10$ is $1$, $4$, $5$, $6$, $5$, $3$, $2$, $1$, $1$, $1$, respectively.  Thus $Q_{d-N}^{(1,-)}(5d-5N+16) = 29$.  Since $d \geq \max\{63,46N-79 \}$, it follows that $d-2N+5 \geq 56$, and thus \eqref{q-lemma} gives that 
\[
q_d^{(1)}(2d-2N+5) \geq \left \lfloor \frac{d-2N+5}{2} \right \rfloor + 1 \geq 29,
\]
which yields \eqref{case2}.

We now prove \eqref{case3}.  Since $d \geq \max\{63,46N-79 \}$, it follows that $x^N_{15} > 7d+13$. Thus any partition counted by $Q_{d-N}^{(1,-)}(7d+13)$ can only use the parts $x^N_i$ with $1 \leq i \leq 14.$  Using that fact that $d \geq \max\{63,46N-79 \}$, one can calculate that the number of partitions of $7d+13$ with largest part $x_i^N$ as $i$ ranges from $1$ to $14$, is at most\footnote{Some variance can occur for certain choices of $d$ and $N$.} $1$, $7$, $12$, $20$, $16$, $18$, $10$, $10$, $5$, $5$, $2$, $2$, $1$, $1$, respectively.  Thus $Q_{d-N}^{(1,-)}(7d+13) \leq 110$.  Since $d \geq \max\{63,46N-79 \}$, it follows that $4d-5N+17 \geq 218$, and thus \eqref{q-lemma} gives that
\[
q_d^{(1)}(5d-5N+17) \geq \left \lfloor \frac{4d-5N+17}{2} \right \rfloor + 1 \geq 110,
\]
which yields \eqref{case3}.
\end{proof}

We now complete the proof of Theorem \ref{biglemon} with the following lemma.  

\begin{lemma}\label{bigbiglemon}
Let $N \geq 2$ and $d \geq \max\{63,46N-79\}$.  Then for all $n \geq 7d+14$, 
\[ 
q_d^{(1)}(n) \geq Q_{d-N}^{(1,-)}(n). 
\]
\end{lemma}

\begin{proof}
We first note that our bound on $d$ allows us to apply Lemma \ref{lem: qd-greater-than-t5d}, so we have the inequality $q_d^{(1)}(n) \geq \rho(T_{5,d};n)$, and thus by \eqref{QtoS} it suffices to show 
\begin{equation}\label{goal}
\rho(T_{5,d};n)\geq \rho(S_d^N; n).
\end{equation} 

Recall that for fixed $d$ and $n$ we write $S^N$ to denote the set of partitions of $n$ with parts in $S_d^N$, and for $\lambda\in S^N$, we let $p_i$ denote the number of times $x_i^N$ occurs as a part in $\lambda$.  Furthermore write $T$ to denote the set of partitions of $n$ with parts in $T_{5,d}$, and for $\mu\in T$, let $q_i$ denote the number of times $y_i$ occurs as a part in $\mu$. Then $|S^N| = \rho(S_d^N; n)$ and $|T| = \rho(T_{5,d}; n)$, so to prove \eqref{goal}, it suffices to construct an injection $\varphi^N : S^N \hookrightarrow T$. 

We decompose $S^N$ into the subsets 
\begin{align}
S_1^N &:=\{\lambda \in S^N \mid p_1 + \alpha \geq (N-2)p_2\}, \label{Definition: S_1} \\
S_2^N &:= \{ \lambda \in S^N \mid p_1 + \alpha < (N-2)p_2 \}, \nonumber
\end{align}
and we further partition $S_2^N$ for integers $\beta \geq 0$ by
\begin{equation}\label{Definition: S_2,beta}
    S_{(2,\beta)}^N := \left\{ \lambda \in S_2^N \mid \beta = \left\lfloor \frac{p_1 + p_5}{d-N-1} \right\rfloor \right\}. 
\end{equation}
By inspection, it is clear that $S^N$ is the disjoint union of the sets $S_1^N$ and $S_{(2,\beta)}^N$ for all $\beta\geq 0$.  Thus we can construct $\varphi^N$ piecewise by constructing injections $\varphi_1^N : S_1^N \hookrightarrow T$ and $\varphi_{(2,\beta)}^N : S_{(2,\beta)}^N \hookrightarrow T$ for each $\beta\geq 0$ that have mutually disjoint images.  To describe such maps, given $\lambda\in S^N$, we define its image in $T$ by specifying the $q_i$ associated to the image in terms of the $p_i$ associated to $\lambda$.   Also, recall by \eqref{alpha} that 
\[
\alpha = \alpha(\lambda) :=\sum_{i \geq 3}(x_i^N-y_i)p_i. 
\]

Define $\varphi_1^N:S_1^N\rightarrow T$ by
\begin{align*}
    q_i=\begin{cases}
    p_1+\alpha-(N-2)p_2, & \text{ if } i=1\\
    p_i, & \text{ if }  i\geq 2.
    \end{cases}
\end{align*}

We first show $\varphi_1^N$ is well defined. Given $\lambda \in S_1^N$, we have by definition of $S_1^N$ that $p_1 + \alpha \geq (N-2)p_2$.  Thus each $q_i\geq 0$ so that $\varphi_1^N(\lambda)$ is indeed a partition into parts from $T_{5,d}$. Furthermore, we see that $\varphi_1^N(\lambda)$ is a partition of $n$, i.e., $\varphi_1^N(\lambda)\in T$, as
\begin{multline*}
\sum_{i\geq 1}q_iy_i = (p_1+\alpha-(N-2)p_2)+p_2(d+2)+\sum_{i\geq 3}p_iy_i = p_1+(d-N+4)p_2+\sum_{i\geq 3}p_ix_i^N = \sum_{i\geq 1}p_ix_i^N = n.
\end{multline*}

To see that $\varphi_1^N$ is injective, suppose $\lambda,\lambda' \in S_1^N$ such that $\varphi_1^N(\lambda)=\varphi_1^N(\lambda')$.  Let $p_i'$ and $q_i'$ denote the number of times $x_i^N$ and $y_i$ occur in $\lambda'$ and $\varphi_1^N(\lambda')$, respectively, and let $\alpha' = \sum_{i\geq 3}(x_i^N - y_i) p_i'$.  Then $q_i=q_i'$ for all $i$ implies that $p_i=p_i'$ for all $i\geq 2$ and $p_1+\alpha-(N-2)p_2=p_1'+\alpha'-(N-2)p_2'$. Since $p_i=p_i'$ for all $i\geq 2$ implies $\alpha=\alpha'$, we have $p_1=p_1'$ and hence that $\lambda=\lambda'$. So $\varphi_1^N : S_1^N \hookrightarrow T$ as desired.

Next, for fixed $\beta \geq 0$, given $\lambda \in S_{(2, \beta)}^N$, let 
\[
\varepsilon=\varepsilon(\lambda) := \begin{cases}
0      & \text{ if } p_2 \text{ is even}, \\
1      &  \text{ if } p_2 \text{ is odd}.
\end{cases}
\]
Then define $\varphi_{(2,\beta)}^N: S_{(2, \beta)}^N \rightarrow T$ by
\begin{align*}
   q_i=\begin{cases}
    p_1+\alpha+\frac{(p_2+\varepsilon)(d-2N-8)}{2}+28\beta+(26+N)\varepsilon, & \text{ if } i=1\\
    2\beta+\varepsilon, & \text{ if } i=2\\
    p_5+\frac{p_2+\varepsilon}{2}-2\beta-2\varepsilon, & \text{ if } i=5\\
    p_i, & \text{ if } i\neq 1,2,5,
    \end{cases}
\end{align*}
To see that $\varphi_{(2,\beta)}^N$ is well defined, we first observe that since $d \geq \max\{63,46N-79\}$, we have easily that $q_i\geq 0$ for all $i \neq 5$. To prove $q_5\geq 0$, it suffices to show that $p_2 - 3\varepsilon \geq 4\beta$.  By the definitions \eqref{Definition: S_2,beta}, \eqref{alpha}, \eqref{Definition: S_1}, as well as $d \geq \max\{63,46N-79\}$, it follows that
\[
4\beta \leq 4 \left(\frac{p_1 + p_5}{d-N-1}\right) \leq 4 \left(\frac{p_1 + \alpha}{d-N-1}\right) < \frac{4(N-2)p_2}{d-N-1} \leq \frac{p_2}{2}.
\]
Moreover, the hypotheses of Lemma \ref{lem: p2-geq-8} are satisfied, so $p_2\geq 8$.  Thus, 
\[
4\beta < \frac{p_2}{2} = p_2 - \frac{p_2}{2} < p_2 - 3 \leq p_2-3\varepsilon.
\]
Thus each $q_i\geq 0$ so that $\varphi_{(2,\beta)}^N(\lambda)$ is indeed a partition into parts from $T_{5,d}$.  Furthermore, we see that $\varphi_{(2,\beta)}^N(\lambda)$ is a partition of $n$, i.e., $\varphi_{(2,\beta)}^N(\lambda)\in T$, as
\begin{align*}
\sum_{i\geq 1}q_iy_i
&=\left( p_1+\alpha+\frac{(p_2+\varepsilon)(d-2N-8)}{2}+28\beta+(26+N)\varepsilon\right)
+(2\beta+\varepsilon)(d+2)\\
& \qquad \qquad +\left(p_5+\frac{p_2+\varepsilon}{2}-2\beta-2\varepsilon\right)(d+16)+\sum_{i\neq 1,2,5}p_iy_i\\
&=p_1+\frac{(p_2+\varepsilon)(2d-2N+8)}{2}+(-d+N-4)\varepsilon+\sum_{i\geq 3}p_ix_i^N\\
&=p_1+p_2(d-N+4)+\sum_{i\geq 3}p_ix_i^N
=\sum_{i\geq 1}p_ix_i^N = n.
\end{align*}

To see that $\varphi_{(2,\beta)}^N$ is injective, suppose $\lambda,\lambda' \in S_{(2, \beta)}^N$ such that $\varphi_{(2, \beta)}^N(\lambda) = \varphi_{(2, \beta)}^N(\lambda')$.  As in the previous case, let $p_i'$ and $q_i'$ denote the number of times $x_i^N$ and $y_i$ occur in $\lambda'$ and $\varphi_{(2, \beta)}^N(\lambda')$, respectively, $\alpha' = \sum_{i\geq 3}(x_i^N - y_i) p_i'$, and also let $\varepsilon'$ denote the residue of $p_2'$ modulo $2$.  Then $q_i=q_i'$ for all $i$ implies that $p_i=p_i'$ for all $i \neq 1,2, 5$ and $\varepsilon = \varepsilon'$.  From $q_1 = q_1'$ and $q_5 = q_5'$, we obtain that
\begin{equation} \label{eq: big-lemma-injective-1}
p_1+(2d-3N-8)p_5+\frac{p_2(d-2N-8)}{2}=p_1'+(2d-3N-8)p_5'+\frac{p_2'(d-2N-8)}{2},
\end{equation}
\begin{equation} \label{eq: big-lemma-injective-2}
p_5+\frac{p_2}{2} = p_5'+\frac{p_2'}{2}.
\end{equation}
Multiplying \eqref{eq: big-lemma-injective-2} by $(d-2N-8)$ and subtracting this from \eqref{eq: big-lemma-injective-1} gives
\begin{equation} \label{eq: big-lemma-injective 3}
    p_1 + (d-N) p_5 = p_1' + (d-N) p_5'.
\end{equation}
From \eqref{Definition: S_2,beta}, we see that $p_1 + p_5 = \beta (d-N-1) + m$ and $p_1' + p_5' = \beta (d-N-1) + m'$, where $0\leq m,m' < d-N-1$. Thus subtracting yields
\begin{equation}\label{mm'}
    (p_1 - p_1') + (p_5 - p_5') = m - m'.
\end{equation}
Combining \eqref{mm'} and \eqref{eq: big-lemma-injective 3} gives
\begin{equation}\label{injective-label-14}
    m' - m = (d-N-1)(p_5' - p_5).
\end{equation}
Since $0\leq m,m' < d-N-1$, \eqref{injective-label-14} implies that $m=m'$ and thus $p_5=p_5'$.  Thus from \eqref{eq: big-lemma-injective-2} it follows that $p_2 = p_2'$, so \eqref{eq: big-lemma-injective-1} yields that $p_1 = p_1'$, and hence $\lambda=\lambda'$. So $\varphi_{(2, \beta)}^N : S_{(2, \beta)}^N \hookrightarrow T$ as desired. 

It remains to show that the images of all of the $\varphi_1^N$ and $\varphi_{(2, \beta)}^N$ are distinct.  First observe that if $\beta\neq \beta'$, $\lambda \in S_{(2, \beta)}^N$, and $\lambda' \in S_{(2, \beta')}^N$, then $\varphi_{(2,\beta)}^N(\lambda) \neq \varphi_{(2,\beta')}^N(\lambda')$ since $q_2 \neq q_2'$.  

Now fix $\beta \geq 0$, and suppose toward contradiction that $\lambda \in S_{(2,\beta)}^N$ and $\lambda' \in S_1^N$ such that $\varphi_{(2,\beta)}^N(\lambda)=\varphi_1^N(\lambda')$. Then $q_i=q_i'$ for all $i$ immediately gives that $p_i=p_i'$ for all $i\neq 1,2,5$ and 

\begin{align*}
p_1' + \alpha' - (N-2)p_2'  &=  p_1+\alpha+\frac{(p_2+\varepsilon)(d-2N-8)}{2}+28\beta+(26+N)\varepsilon, \\
p_2'  &=  2\beta+\varepsilon, \\
p_5'  &=  p_5+\frac{p_2+\varepsilon}{2}-2\beta-2\varepsilon, 
\end{align*}
which yield that
\begin{multline}\label{eq: general-big-lemma-injective 2}
    p_1'+(2d-3N-8)p_5'=p_1+(2d-3N-8)p_5+\frac{(p_2+\varepsilon)(d-2N-8)}{2}+(2N+24)(\beta+\varepsilon),
\end{multline}
\begin{equation}\label{eq: general-big-lemma-injective 3}
    p_5'=p_5+\frac{p_2+\varepsilon}{2}-2(\beta+\varepsilon).
\end{equation}
Multiplying \eqref{eq: general-big-lemma-injective 3} by $(2d-3N-8)$ and subtracting this from \eqref{eq: general-big-lemma-injective 2} gives
\begin{equation}\label{eq: general-big-lemma-injective 4}
    p_1'=p_1+\frac{(p_2+\varepsilon)(N-d)}{2}+(4d-4N+8)\beta+(4d-4N+8)\varepsilon.
\end{equation}
From \eqref{Definition: S_1} and \eqref{Definition: S_2,beta} we have 
\begin{align}\label{ineq1}
p_1 & \leq p_1+\alpha <(N-2)p_2, \\
\beta & \leq \frac{p_1+p_5}{d-N-1}\leq \frac{p_1+\alpha}{d-N-1}<\frac{(N-2)p_2}{d-N-1}. \nonumber
\end{align}
Thus, \eqref{eq: general-big-lemma-injective 4} and \eqref{ineq1} yield that
\begin{align}\label{ineq2}
p_1'& < (N-2)p_2+\frac{(p_2+\varepsilon)(N-d)}{2}+(4d-4N+8)\left(\frac{(N-2)p_2}{d-N-1}+ \varepsilon \right)  \nonumber \\
&=\frac{(-d^2+(12N-19)d-11N^2+33N-28)p_2+(7d^2+d(-14N+9)+7N^2-9N-16)\varepsilon}{2d-2N-2}.
\end{align}

Since the hypotheses of Lemma \ref{lem: p2-geq-8} are satisfied, we have that $p_2\geq 8$.  If $p_2=8$, then $\varepsilon=0$ and \eqref{ineq2} becomes
\[
p_1' <  
\frac{-4d^2+(48N-76)d -44N^2+132N-112}{d-N-1}.
\]
Since $d \geq \max\{63,46N-79\}$, the denominator is always positive.  But when $d > \frac{12N-19+\sqrt{100N^2-324N+249}}{2}$, the numerator is negative, which would yield a contradiction since $p_1'\geq 0$.  Since $100N^2-324N+249 < (10N-16)^2$, it suffices to show that $d \geq 11N - 17$, which follows easily from the fact that $d \geq \max\{63,46N-79\}$.  Thus we have a contradiction in the case when $p_2=8$.

Suppose $p_2\geq 9$. Since $d \geq \max\{63,46N-79\}$, for all $N\geq 2$ we have
\begin{align*}
-d^2+d(12N-19)-11N^2+33N-28 & \leq 0, \\
7d^2+d(-14N+9)+7N^2-9N-16 & \geq 0.
\end{align*}
Thus \eqref{ineq2} yields that
\[
p_1' \leq \frac{-d^2+(47N-81)d-46N^2+144N-134}{d-N-1}.
\]
As above, when $d > \frac{47N-81+\sqrt{2025N^2-7038N+6025}}{2}$ the right hand side is negative which contradicts the nonnegativity of $p_1'$.  Since $2025N^2-7038N+6025 < (45N-78)^2$, it suffices to show that $d \geq 46N - 79$, which is immediate from our bound $d \geq \max\{63,46N-79\}$.  Thus we have a contradiction in the case when $p_2\geq 9$, and we have shown that $\varphi_{(2,\beta)}^N(\lambda) \neq \varphi_1^N(\lambda')$ for any $\lambda \in S_{(2,\beta)}^N$ and $\lambda' \in S_1^N$.

Thus $\varphi_1^N$ and $\varphi_{(2,\beta)}^N$ for each $\beta\geq 0$ together form a piecewise injective map $\varphi^N : S^N \hookrightarrow T$, which gives our desired inequality. 
\end{proof}

\section{Proof of Theorem \ref{thm:gen_kp} and Corollary \ref{thm:gen_dkst}}\label{conclusion}
We now demonstrate that the methods of Inagaki and Tamura \cite{Inagaki-Tamura} together with Corollary \ref{littlelemon} yield the generalized Kang-Park type result given in Theorem \ref{thm:gen_kp}. 

\begin{proof}[Proof of Theorem \ref{thm:gen_kp}]
We first suppose that $n\geq d+2a$.  Write $\hat{n}_a$ and $\hat{d}_a$ to denote the least nonnegative residue of $-n$ and $-d$ modulo $a$, respectively, so that $\lceil\frac{n}{a}\rceil=\frac{n+\hat{n}_a}{a}$ and $\lceil\frac{d}{a}\rceil=\frac{d+\hat{d}_a}{a}$.  Then using Lemma \ref{Lem: st-thm with ceiling}, Corollary \ref{littlelemon}, and Lemma \ref{a-to-1}, we obtain
\begin{equation*}
    q_d^{(a)}(n)\geq q_{\frac{d+\hat{d}_a}{a}}^{(1)}\left(\frac{n+\hat{n}_a}{a}\right)\geq Q_{\frac{d+\hat{d}_a}{a}-4}^{(1,-)}\left(\frac{n+\hat{n}_a}{a}\right)=Q_{d+\hat{d}_a-a-3}^{(a,-)}\left(n+\hat{n}_a\right).
\end{equation*}
Thus it remains to show that
\begin{equation}\label{lastineq}
    Q_{d+\hat{d}_a-a-3}^{(a,-)}(n+\hat{n}_a)\geq Q_d^{(a,-)}(n).
\end{equation}

Define
\begin{align*}
&S:=\{x\in\mathbb{N}\mid x\equiv \pm a \!\!\!\!\pmod{d+3}\}\setminus\{d+3-a\},\\
    &T:=\{x\in\mathbb{N}\mid x\equiv \pm a \!\!\!\!\pmod{d+\hat{d}_a-a}\}\setminus\{d+\hat{d}_a-2a\},
\end{align*}
and observe that $Q_d^{(a,-)}(n) = \rho(S;n)$ and $Q_{d+\hat{d}_a-a-3}^{(a,-)}(n+\hat{n}_a) = \rho(T;n+\hat{n}_a)$. Letting $x_i$ and $y_i$ denote the $i^{\text{th}}$ smallest elements of $S$ and $T$, respectively, we have that $x_1=y_1=a$, and
\[
\begin{tabular}{rrl}
$x_{2i} = i(d+3) + a$, & $y_{2i} = i(d+\hat{d}_a-a) + a$, & for $i\geq 1$,  \\
$x_{2i-1} = i(d+3) - a$, & $y_{2i-1} = i(d+\hat{d}_a-a) - a$, & for $i\geq 2$.\\
\end{tabular}
\]
Clearly $a\mid y_i$ for all $i\geq 1$, and moreover, $x_i\geq y_i$ for all $i\geq 1$ since $0 \leq \hat{d}_a < a$.  Thus by Lemma \ref{IT-modiefied-ST}, we have \eqref{lastineq} as desired.

We now consider $1\leq n\leq d+2a-1$. As in the proof of Lemma \ref{bigsmalllemon}, we observe that $q_d^{(a)}(n)$ is a weakly increasing function, however $Q_d^{(a,-)}(n)$ is not.    

If $1\leq n\leq a-1$, then $q_d^{(a)}(n)=0=Q_d^{(a,-)}(n)$.  Also, $q_d^{(a)}(a) = 1$ and $Q_d^{(a,-)}(n)\leq 1$ for all $a\leq n \leq d+a+2$ since $a$ is the only available part.  Thus it remains to consider when $d+a+3 \leq n \leq d+2a-1$, which only occurs for $a\geq 4$.

By our hypothesis that $\left\lceil\frac{d}{a}\right\rceil \geq 105$, it follows that $d+2a-1 < 2d-a+6$.  Thus the only available parts for a partition counted by $Q_d^{(a,-)}(n)$ when $d+a+3 \leq n\leq d+2a-1$ are $a$ and $d+a+3$.  Furthermore, the part $d+a+3$ can occur at most once since $2d+2a+6 > d+2a-1$.  Thus a partition counted by $Q_d^{(a,-)}(n)$ when $d+a+3 \leq n\leq d+2a-1$ is either a sum of parts of size $a$, which can only occur when $n\equiv 0 \pmod{a}$, or $d+a+3$ plus a sum of parts of size $a$, which can only occur when $n\equiv d+3 \pmod{a}$.  Thus $Q_d^{(a,-)}(n)\leq 1\leq q_d^{(a)}(n)$ except when $d \equiv -3 \pmod{a}$ and $n\equiv 0 \pmod{a}$ simultaneously.  But if $d=ka-3$ for $k\geq 1$, then $(k+1)a \leq n \leq (k+2)a - 4$, so the only exception occurs when $n=d+a+3$.
\end{proof}

We now prove Corollary \ref{thm:gen_dkst}.

\begin{proof}[Proof of Corollary \ref{thm:gen_dkst}]
By definition, $\Delta_d^{(a,-,-)}(n) \geq \Delta_d^{(a,-)}(n)$, since there are fewer parts available for partitions counted by $\Delta_d^{(a,-,-)}(n)$.  Thus by Theorem \ref{thm:gen_kp}, we have $\Delta_d^{(a,-,-)}(n) \geq 0$ for any $a,d \geq 1$ such that $\left\lceil\frac{d}{a}\right\rceil \geq 105$ and $n\geq1$, except possibly when $d\equiv -3 \pmod{a}$ and $n=d+a+3$.  However in these cases, observe that $Q_d^{(a,-,-)}(d+a+3) = 1$, since $d+a+3$ is the only available part by definition.  Also, $q_d^{(a)}(d+a+3) \geq 1$ since $d+a+3$ is a partition counted by $q_d^{(a)}(d+a+3)$.  Thus $q_d^{(a)}(n) \geq   Q_d^{(a,-,-)}(n)$ in all of our considered cases.
\end{proof}
 
\section{Concluding Remarks}

By work of Kang and Kim\footnote{Note that Kang and Kim use different notation that what we are using here.} \cite[Thm. 1.1]{KANG-KIM} and the fact that $Q_{d}^{(a)}(n) \geq Q_{d}^{(a,-)}(n)$, it follows that when $\gcd(a, d-N) = 1$,
\[
\lim_{n \rightarrow \infty}(q_d^{(a)}(n) - Q_{d-N}^{(a,-)}(n)) = \infty,
\]
for all $N < d+3- \lfloor \frac{\pi^2}{3A_d}\rfloor$, where $A_d = \frac{d}{2}\log^2{\alpha_d} + \sum_{r=1}^\infty r^{-2}\alpha_d^{rd}$, with $\alpha_d$ the unique real root of $x^d +x - 1$ in the interval $(0, 1)$.  Thus it may be possible to generalize Theorem \ref{biglemon} to an inequality of the form $q_d^{(a)}(n) \geq Q_{d-N}^{(a,-)}(n)$ for more general $a$.

\section*{Acknowledgements}
We thank Ryota Inagaki and Ryan Tamura for their helpful correspondence and interesting idea.

\end{document}